\documentclass[11pt]{amsart}

\usepackage{amsmath}
\usepackage{amsthm}
\usepackage{thmtools}
\usepackage{amsfonts}
\usepackage{graphicx}
\usepackage{subfigure}
\usepackage{hyperref}
\usepackage{xspace}
\usepackage{parskip}
\usepackage{color}
\usepackage{tikz}
\usepackage{verbatim}
\usepackage{xspace}
\usepackage{mdframed}
\usepackage{mathtools}

\graphicspath{{figures/}}

\makeatletter
\def\thm@space@setup{%
  \thm@preskip=\parskip \thm@postskip=0pt
}
\makeatother

\declaretheorem[parent=section]{lemma}
\declaretheorem[sibling=lemma]{theorem}
\declaretheorem[sibling=lemma, name=Proposition]{prop}

\declaretheorem[sibling=lemma]{corollary}

\newcommand{\boundary}{\partial}
\newcommand{\set}[1]{\left\{#1\right\}}
\newcommand{\allsuchthat}[2]{\left\{#1\colon#2\right\}}
\newcommand{\closure}[1]{\overline{#1}}
\newcommand{\scc}{simple closed curve\xspace}
\newcommand{\sccs}{simple closed curves\xspace}

\newcommand{\co}{\colon}
\newcommand{\R}{\mathbb{R}}

\newcommand{\defn}[1]{\emph{#1}}

\renewcommand{\L}{\mathcal{L}}
\newcommand{\Lhat}{\widehat{L}}

\newcommand{\AC}{\mathcal{AC}}
\newcommand{\ACm}{\dot{\mathcal{AC}}}
\newcommand{\dm}{\dot{d}}
\newcommand{\dist}[2]{d(#1,#2)}
\newcommand{\gprod}[3]{(#1,#2)_{#3}}
\newcommand{\seg}[2]{[#1,#2]}

\newcommand{\translen}[1]{\left|#1\right|}
\newcommand{\geomint}[2]{\iota(#1,#2)}

\newcommand{\Nbf}{\mathbf{N}}
\newcommand{\Sbf}{\mathbf{S}}
\newcommand{\Wbf}{\mathbf{W}}
\newcommand{\Ebf}{\mathbf{E}}

\newcommand{\cf}{cf.\xspace}

\begin{document}
\title{Bridge numbers of knots in the page of an open book}
\author{R.\ Sean Bowman\and Jesse Johnson}
\address{Oklahoma State University, Stillwater, OK 74078}
\email{r.sean.bowman@gmail.com}
\email{jjohnson@math.okstate.edu}
\date{\today}

\begin{abstract}
  Given any closed, connected, orientable $3$--manifold and integers
  $g\geq g(M), D > 0$, we show the existence of knots in $M$ whose genus $g$
  bridge number is greater than $D$.  These knots lie in a page of an
  open book decomposition of $M$, and the proof proceeds by examining
  the action of the map induced by the monodromy on the arc and curve
  complex of a page.  A corollary is that there are Berge knots of
  arbitrarily large genus one bridge number.
\end{abstract}

\maketitle
\section{Introduction}

Let $M$ be a closed, connected, orientable $3$--manifold.  An
\defn{open book decomposition} of $M$ is a pair $(L, \pi)$ where
$L\subseteq M$ is a link and $\pi\co M\setminus L\to S^1$ is a surface
bundle map such that the closure of each \defn{page}
$F^s=\closure{\pi^{-1}(s)}$ is an embedded surface with boundary $L$.
Each page is homeomorphic to an abstract surface $F$, and the link
exterior $\closure{M\setminus N(L)}$ can be identified with the quotient
$F\times [0,1]$ by a homeomorphism $\phi\co F\to F$ called the
\defn{monodromy}.  The closure of the union of two pages $\Sigma =
F^{s_0} \cup F^{s_1}$ (for $s_0 \neq s_1 \in S^1$) is a Heegaard
surface for $M$ because each component of the complement is
homeomorphic to a handlebody $F^s \times [0,1]$. Moreover, Berge noted
that a nonseparating essential simple closed curve $K \subseteq F^s$
in one of the fibers will be primitive in both handlebodies. Such a
knot is called \defn{doubly primitive}, and Berge showed that doubly
primitive knots in genus two Heegaard surfaces have lens space
surgeries. In fact, essential \sccs in the Seifert surfaces of the
trefoil and figure eight knots (both of which are fibered, and
therefore give rise to open books of $S^3$) are two of Berge's famous
families of knots~\cite{Berge}. In general, a doubly primitive knot in
a genus $g$ Heegaard surface will have a Dehn surgery producing a
$3$--manifold admitting a genus $g-1$ Heegaard surface.

A \defn{bridge surface} for a link $L \subseteq M$ is a Heegaard
surface $\Sigma$ for $M$ such that the intersection of $L$ with each
of the two handlebodies in the complement of $\Sigma$ is a collection
of boundary parallel arcs. The \defn{bridge number} of a bridge
surface is the number arcs of intersection (i.e.\ bridges) with each
handlebody and the \defn{genus $g$ bridge number} $b_g(L)$ of $L$ is
the minimal bridge number among all the genus $g$ bridge surfaces for
$L$.  

In the present paper, we prove a general result for knots in the page
of an open book decomposition of any closed $3$--manifold.

\begin{theorem}\label{theorem:main}
  Let $(L, \pi)$ be an open book decomposition of a closed, connected,
  orientable $3$--manifold $M$ with page $F$ such that $F$ is not a
  disk, annulus, or pair of pants.  For any integer $D > 0$ there are
  infinitely many knots $K\subseteq F$ such that
  \[ b_g(K) > D \]
  for every $g(M)\leq g \leq -\chi(F)$.  The exteriors of these knots
  have Heegaard genus $1-\chi(F)$, and when $M=S^3$ we may choose the
  knots to be hyperbolic.
\end{theorem}

Here $g(M)$ is the Heegaard genus of $M$, the minimum genus over all
Heegaard surfaces for $M$.  As stated earlier, knots in the fiber of a
fibered link are primitive on both sides of the natural genus
$1-\chi(F)$ splitting given by two copies of the fiber.  This means
that the exteriors of these knots have Heegaard splittings of Hempel
distance at most two.  Contrast the following corollary with theorems
of Minsky-Moriah-Schleimer~\cite{MinskyEtAl} and
Moore-Rathbun~\cite{MooreRathbun} which exhibit knots in $S^3$ and an
arbitrary closed, orientable $3$--manifold, respectively, that have
high bridge number at many genera.

\begin{corollary}\label{cor1}
  Let $M$ be a closed, connected, orientable $3$--manifold admitting
  an open book decomposition with pages of Euler characteristic $-k$,
  $k>0$, which are not $3$--punctured spheres.  Then for any integers
  $g\geq k$ and $D>0$, there are infinitely many knots $K\subseteq M$
  such that
  \begin{enumerate}
  \item $K$ has a nontrivial surgery yielding a manifold of Heegaard
    genus at most $g$,
  \item $b_{g'}(K)>D$ for every $g(M)\leq g'\leq g$, and
  \item $\closure{M\setminus N(K)}$ has a minimal genus Heegaard
    splitting of distance at most two and genus $g+1$.
  \end{enumerate}
\end{corollary}
\begin{proof}
  Every closed, connected, orientable $3$--manifold admits an open
  book decomposition (see~\cite{Etnyre06} for several proofs).  The
  operation of stabilization, or plumbing with a Hopf band, decreases
  the Euler characteristic of the page by one.  Therefore we may
  choose an open book decomposition of $M$ with pages $F$ satisfying
  $-\chi(F)= g$ for any given $g\geq k$.
  Applying~\autoref*{theorem:main}, we obtain knots with large bridge
  number for every $g(M)\leq g'\leq g$.

  As noted above, $K$ lies in a genus $1-\chi(F)$ splitting surface
  $\Sigma$ so that it is primitive on both sides.  By Berge's
  construction~\cite{Berge}, surgery at the surface slope yields a
  manifold of Heegaard genus at most $g$.  Furthermore, isotoping
  $\Sigma$ off $K$ so that $K$ ends up in one of the two handlebodies
  $\Sigma$ bounds creates a Heegaard surface defining a Heegaard
  splitting of the knot exterior $\closure{M\setminus N(K)}$. This
  splitting will have distance at most two by construction.  The
  exterior of $K$ has no splittings of smaller genus because of the
  previous bridge number bounds.
\end{proof}

Finally, we note that there are Berge knots with arbitrarily large
genus one bridge number.  This follows directly by
applying~\autoref*{theorem:main} to knots in the fiber of the trefoil or
figure eight knot.  (See~\cite{BowmanEtAl} for another argument.)
Note that this fact has been known to Baker for some time~\cite{Baker}.

\begin{corollary}
  There are Berge knots of type VII and VII, knots which lie in the
  fiber of the trefoil or figure eight, respectively, with arbitrarily
  large genus one bridge number.
\end{corollary}

\textbf{Acknowledgments} We would like to thank Ken Baker for
sharing his knowledge of bridge numbers of Berge knots.  We would also
like to thank Scott Taylor for helpful conversations on
arguments in~\cite{BlairEtAl}.

\section{Definitions}
Let $M$ be a compact, connected, orientable $3$--manifold whose
boundary is either empty or a union of one or more
tori. (\autoref*{theorem:main} refers to the case when $\partial M =
\emptyset$, but we will need the more general case for a number of
steps in the argument.)  We write $|X|$ to denote the number of
components of a manifold $X$.  For a submanifold $L$ of $M$, we write
$N(L)$ for a closed regular neighborhood of $L$ in $M$.  Let
$F\subseteq M$ be a properly embedded surface and $L$ a properly
embedded $1$--manifold. We will write $F_L$ to mean
$\closure{F\setminus N(L)}$, and similarly $M_L$ to mean
$\closure{M\setminus N(L)}$.  An \defn{essential curve} in $F_L$ is a
\scc that does not bound a disk in $F_L$ and is not parallel to a
boundary component of $F_L$.  A disk $D$ embedded in $M_L$ is a
\defn{compressing disk} for $F_L$ if $D\cap F = \boundary D$ is an
essential \scc in $F_L$.  The surface $F$ will be called
\defn{incompressible} if there are no compressing disks for $F$.

An arc properly embedded in $F_L$ is called \defn{essential} if it is
not parallel in the surface to a subarc of $\boundary F_L$.  We say an
incompressible surface $F_L$ is \defn{$\boundary$--compressible} if
there is a disk $D$ in $M_L$ so that $\alpha=D\cap F_L$ is an
essential arc in $F_L$, $\beta=D\cap\boundary M_L$ is an arc in
$\boundary D$, $\boundary D=\alpha\cup\beta$, and $\alpha\cap\beta =
\boundary\alpha = \boundary\beta$.  Otherwise $F_L$ is
\defn{$\boundary$--incompressible}.  When $F_L$ is incompressible,
$\boundary$--incompressible, not parallel to a component of $\boundary
M_L$, and not a sphere bounding a ball disjoint from $L$, we say that
$F_L$ is \defn{essential}.

The \defn{arc and curve complex} $\AC(F)$ of $F$, is a simplicial
complex whose vertices represent isotopy classes of essential \sccs or
arcs properly embedded in $F$, modulo isotopy.  When $g(F)>1$,
$g(F)=1$ and $|\boundary F|>0$, or $g(F)=0$ and $|\boundary F| > 3$,
two vertices bound an edge if they have disjoint representatives in
$F$.  We will only be concerned with the one skeleton of $\AC(F)$ in
the present work.  The distance $d(u, v)$ between vertices in $\AC(F)$
is the number of edges in the shortest path from $u$ to $v$.
Similarly, if $U$ and $V$ are subsets of vertices of $\AC(F)$, $d(U,
V)$ is the shortest path from a vertex of $U$ to a vertex of $V$.

Since the monodromy of a surface bundle may twist an arc around a
boundary component, we need to keep track of arcs up to isotopy fixing
the boundary.  Choose a collection of points $m\subseteq\boundary F$,
one in each component of $\boundary F$.  Define the \defn{marked arc
  and curve complex} of $F$, $\ACm(F)$, to be the simplicial complex
whose vertices represent isotopy classes of essential \sccs or arcs
properly embedded in $F$ and disjoint from $m$, modulo isotopy
disjoint from $m$.  As before, when $g(F)>1$, $g(F)=1$ and $|\boundary
F|>0$, or $g(F)=0$ and $|\boundary F| > 3$, two vertices bound an edge
if they have disjoint representatives in $F$.  We denote the metric on
$\ACm(F)$ by $\dm$, and define $\dm(U,V)$ for subsets of vertices as
above. 

Each isotopy class defining a vertex of $\ACm(F)$ is contained in an
isotopy class that defines a vertex of $\AC(F)$, so there is a natural
map $p$ from the vertices of $\ACm(F)$ to $\AC(F)$.  This map takes
simplices to simplices (though it may take a given simplex in
$\ACm(F)$ to a lower dimensional simplex in $\AC(F)$), so $p$ extends
to a simplicial map between the two simplicial complexes.

Although we do not use the following lemma in this paper, it clarifies
the relationship between $\ACm(F)$ and the more commonly encountered
$\AC(F)$.

\begin{lemma}\label{lem:marked-complex-qi}
  The map $p$ defines a quasi-isometry from $\ACm(F)$ to $\AC(F)$ with
  multiplicative constant one.  In fact, for any pair of vertices $a,
  b$ in $\ACm(F)$,
  \[ \dm(a, b) - 2 \leq d(p(a), p(b)) \leq \dm(a, b). \]
\end{lemma}
\begin{proof}
  Let $a$ and $b$ be vertices in $\ACm(F)$.  Because the map $p$ takes
  edges and vertices of $\ACm(F)$ to edges and vertices of $\AC(F)$,
  it is immediate that $d(p(a), p(b))\leq \dm(a, b)$. 

  For the second inequality, note that the map $p$ is one-to-one on the
  vertices of $\AC(F)$ represented by \sccs, but infinite-to-one on
  the vertices represented by arcs.  Let $v_0,v_1,\dots,v_n$ be a path
  in $\AC(F)$ such that $v_0=p(a)$, $v_n=p(b)$, and $n=d(p(a), p(b)$.  

  If $v_1$ is a \scc in $F$, then define $a_1$ to be a vertex in
  $\ACm(F)$ represented by this \scc.  Otherwise, choose
  $\alpha\subseteq F$ to be a representative of $a$.  Since
  $p(a)=v_0$, $\alpha$ is also a representative of $v_0$.  Moreover,
  because $v_0$ and $v_1$ cobound an edge in $\AC(F)$, we can choose a
  representative $\beta$ for $v_1$ that is disjoint from $\alpha$.  In
  fact, we can choose $\beta$ so that its endpoints are disjoint from
  $m$, the marked points in the boundary of $F$.  Thus $\beta$ defines
  a vertex $a_1$ in $\ACm(F)$.  By construction $a_0$ and $a_1$ will
  cobound an edge in $\ACm(F)$.

  We can repeat this construction for each successive value $j\leq n$
  to find a vertex $a_j$ such that $p(a_j)=v_j$ and $v_j, v_{j-1}$
  cobound an edge in $\ACm(F)$.  

  For the final vertex $v_n$, we find that $p(a_n)=v_n=p(b)$.  Let
  $b'$ be a vertex represented by an essential \scc in $F$ disjoint
  from the arc of \scc representing $b$.  Because this \scc is
  disjoint from the boundary, it will also be disjoint from a
  representative for $a_n$.  Thus the sequence $a, a_1, a_2,\dots,
  a_n, b', b$ defines a path in $\ACm(F)$ of length $n+2$, so $\dm(a,
  b)\leq d(p(a), p(b)) + 2$.  Combining this with the previous
  inequality completes the proof.
\end{proof}


Work of Masur and Minsky~\cite{MasurMinsky} implies that $\ACm(F)$ is
a $\delta$--hyperbolic space in the sense of Gromov.  The monodromy
map $\phi$ of an open book induces an isometry of $\ACm(F)$ which we
will also denote by $\phi$.  The distance between a point
$x\in\ACm(F)$ and its image is called the \defn{translation distance}
of $x$ under the isometry.  Bachman and
Schleimer~\cite{BachmanSchleimer} give a bound on the translation
distance of the action of a surface bundle monodromy on the curve
complex of a fiber, and we provide a similar bound for an open book
monodromy acting on $\ACm(F)$.  

A crucial concept used in the present paper is that of the \defn{axis}
of $\phi$, $A_{\phi}\subseteq\ACm(F)$, roughly the set of points of
$\ACm(F)$ that have sufficiently small translation distance under
$\phi$ (see~\autoref*{section:curve-complex}).  We show that when
$\phi\co F\to F$ is pseudo-Anosov, this axis behaves similarly to the
axis of a hyperbolic isometry of hyperbolic $2$--space.

Let $\Sigma$ be a genus $g$ Heegaard splitting of $M$, where $0 \leq g
\leq -\chi(F)$.  We identify $F$ with $F^0$ and let $K$ be a closed
curve in $F$. The curve $K$ can be viewed from two perspectives:
first, as a loop in $F$, it defines a vertex in the marked arc and
curve complex. Second, since $F$ is embedded in $M$, $K$ defines a
knot in $M$. The proof of \autoref*{theorem:main} is based on a
comparison between these two views of $K$.

\section{Isometries of $\ACm(F)$}
\label{section:curve-complex}
In this section we consider $K$ as a vertex of $\ACm(F)$ and examine
its image under an isometry induced from a surface automorphism.  We
define the axis of such an isometry and show that points close
to the axis have small translation distance and, conversely, that
points far from the axis have large translation distance.  Finally, we
relate the distance between two vertices in $\ACm(F)$ to the geometric
intersection number of representatives of those vertices.

Let $(X,d)$ be a metric space with $X$ infinite. In our case, $X$ will
be the $1$--skeleton of the arc and curve complex $\ACm(F)$, with the
nonstandard convention of including points of the edges in the metric
space as well as the vertices. Recall that we have defined
\[ d(A, B) = \inf\allsuchthat{d(a, b)}{a\in A, b\in B}. \] 

A \defn{geodesic path} from $x\in X$ to $y\in X$ is a map $c$ from a
closed interval $[0,l]\subseteq\R$ to $X$ such that $c(0)=x$,
$c(l)=y$, and $d(c(t), c(t'))=|t-t'|$ for every $t,t'\in [0,l]$.  The
image of $c$ is called a \defn{geodesic segment} or \defn{arc}.  When
the choice of geodesic segment connecting two points $x,y\in X$ does
not matter, we denote it by $[x,y]$.  Similarly, a \defn{geodesic ray}
is an isometric embedding of the interval $[0,\infty)$ to $X$, and a
\defn{geodesic line} is an isometric embedding of $\R$ to $X$.

Define the \defn{Gromov product}
\[ \gprod{x}{y}{w} = \frac{1}{2}\left(\dist{x}{w} 
            + \dist{y}{w} - \dist{x}{y}\right). \]

We say that the space $(X, d)$ is \defn{$\delta$--hyperbolic} if
triangles in $X$ are \defn{$\delta$--thin}: each side is contained in
the $\delta$ neighborhood of the other two.  Equivalently,
\[ \dist{x}{q} + \dist{y}{p} \leq \max\{\dist{x}{y} + \dist{p}{q},
  \dist{x}{p} + \dist{y}{q}\} + 2\delta \] for any points $x, y, p,
  q\in X$ (see~\cite[III.H.1.20]{Bridson99}).

Given a closed subset $A\subseteq X$ and a point $p\in X\setminus A$,
define a \defn{projection} of $x$ to $A$ to be a point $p\in A$ such
that $d(x,p) = d(x, A)$.

\begin{lemma}\label{side-estimate}
  Let $x\in X$ and let $p$ be a projection of $x$ to a geodesic
  $\tau$.  Then for every $q\in\tau$,
  \[ \gprod{x}{q}{p}\leq 4\delta \]
  and so
  \[ \dist{x}{q}\geq\dist{x}{p}+\dist{p}{q} - 8\delta. \]
\end{lemma}
\begin{proof}
  Consider the geodesic triangle with vertices $x$, $p$, and $q$, as
  on the left in \autoref*{fig:geometry}.
  Following~\cite[III.H.1.17]{Bridson99}, we may choose points
  $y_1\in\seg{x}{p}$ and $y_2\in\seg{p}{q}$ such that
  $\dist{y_1}{p}=\gprod{x}{q}{p}$ and $\dist{y_1}{y_2}\leq 4\delta$.
  We have $\dist{x}{y_1} = \dist{x}{p}-\dist{y_1}{p}$.  Furthermore,
  $\dist{x}{p}\leq\dist{x}{y_2}$ since $p$ is a projection of $x$ to
  $\tau$.  Combining these estimates and using the triangle
  inequality, we get
  \begin{align*}
    \dist{x}{p} & \leq \dist{x}{y_2}\\
    & \leq \dist{x}{y_1} + \dist{y_1}{y_2}\\
    & \leq \dist{x}{p} - \dist{y_1}{p} + 4\delta.
  \end{align*}
  Therefore $\dist{y_1}{p}\leq 4\delta$.  Finally, recall that $d(y_1,
  p) = \gprod{x}{q}{p}$.
\end{proof}

\begin{figure}[htb]
  \begin{center}
  \includegraphics[width=4.5in]{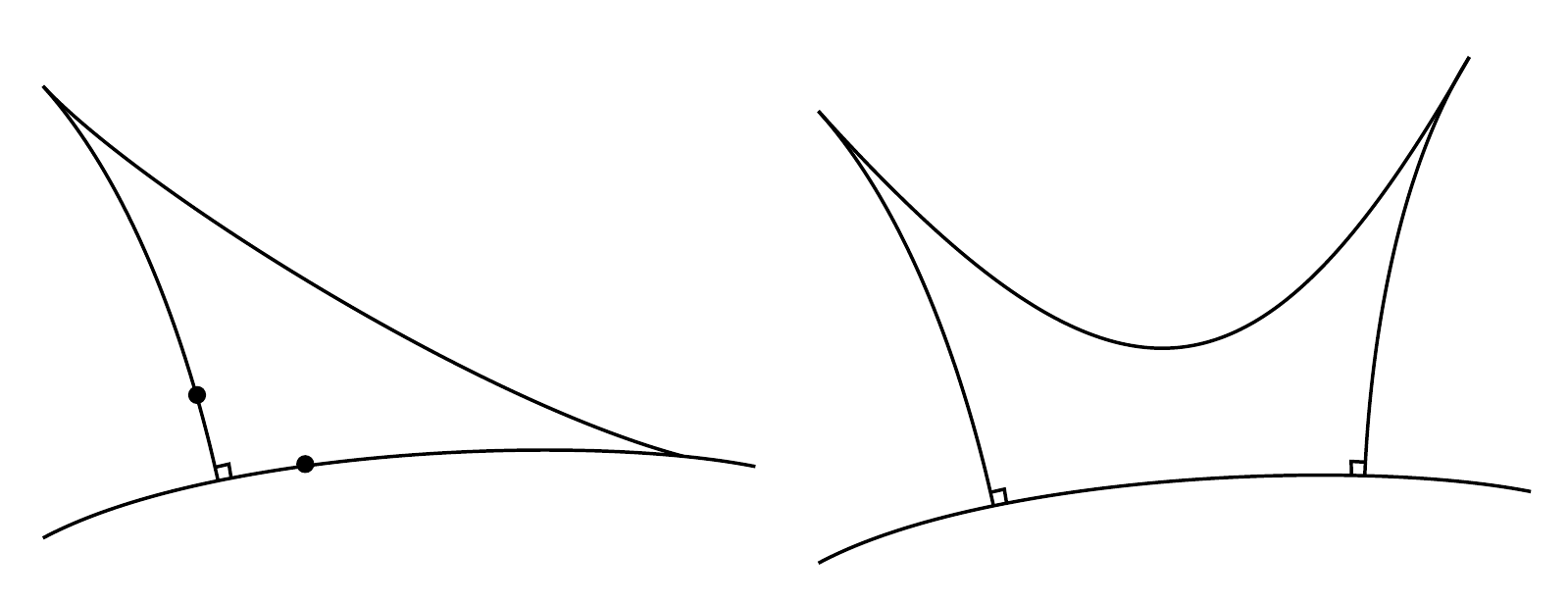}
  \put(-165,100){$x$}
  \put(-18,114){$y$}
  \put(-122,12){$p$}
  \put(-44,17){$q$}
  \put(-325,105){$x$}
  \put(-282,16){$p$}
  \put(-190,21){$q$}
  \put(-297,41){$y_1$}
  \put(-261,19){$y_2$}
  \caption{Triangle and quad used in \autoref*{side-estimate} 
    and \autoref*{quad-estimate}.}\label{fig:geometry}
  \end{center}
\end{figure}

\begin{lemma}\label{quad-estimate}
  Let $p,q\in X$ be projections of points $x, y\in X$ to a geodesic
  $\tau$, and suppose that $d(p,q)>9\delta$.  Then
  \[ \dist{x}{y}\geq\dist{x}{p} + \dist{p}{q} + \dist{q}{y} -
  18\delta. \]
\end{lemma}
\begin{proof}
  Consider the quadrilateral with vertices $x, y, p$, and $q$ as on the
  right of \autoref*{fig:geometry}.  By~\autoref*{side-estimate} we
  have that both $\gprod{x}{q}{p}$ and $\gprod{y}{p}{q}$ are not
  greater than $4\delta$.  By the definition of the Gromov product,
  this yields
  \[ \dist{x}{q}  \geq \dist{x}{p} + \dist{p}{q} - 8\delta \]
  and 
  \[ \dist{y}{p}  \geq \dist{y}{q} + \dist{q}{p} - 8\delta. \]
  Combining these two inequalities we see that
  \begin{equation}\label{eq:quad-inequality}
    \dist{x}{p}+\dist{y}{q} + 2\dist{p}{q} - 16\delta 
           \leq \dist{x}{q} + \dist{y}{p}.
  \end{equation}
  Since $X$ is $\delta$--hyperbolic, we have
  \[ \dist{x}{q} + \dist{y}{p} \leq \max\{\dist{x}{y} + \dist{p}{q},
  \dist{x}{p} + \dist{y}{q}\} + 2\delta. \] 
  If the second argument of
  $\max$ is greater, we obtain using the two previous inequalities
  $\dist{p}{q}\leq 9\delta$, a contradiction.  Therefore we must have 
  \[ \dist{x}{q} + \dist{y}{p} \leq \dist{x}{y} + \dist{p}{q} +
  2\delta, \]
  and so we obtain the conclusion by using~\autoref*{eq:quad-inequality}.
\end{proof}


The \defn{Gromov boundary} $\boundary X$ of a metric space $X$ is the
set of equivalence classes of geodesic rays, modulo the relation that
two rays are equivalent if they stay a bounded distance apart.
Let $\phi$ be an isometry of $X$.  Such isometries are classified into
three types: \defn{elliptic} isometries have bounded orbits,
\defn{parabolic} isometries fix one point in $\boundary X$, and
\defn{hyperbolic} isometries fix two points in $\boundary X$.
Isometries of curve complexes induced from surface automorphisms are
always either elliptic or hyperbolic~\cite{CassonBleiler}, and so we
assume from now on that $\phi$ has this property.

Define the \defn{translation length} of $\phi$ as
\[ \translen{\phi} = \inf_{x\in X}d(x, \phi(x)), \]
and the \defn{axis} of $\phi$, $A_{\phi}$, to be the set of all points
$x\in X$ for which
\[ d(x, \phi(x)) \leq \max\{\translen{\phi}, 10\delta\}. \] 

This definition roughly follows~\cite{Coulon}, as does the proof
of~\autoref*{translation-distance} below. Note that while the term
``axis'' often suggests hyperbolic isometries in which the axis is a
neighborhood of a bi-infinite geodesic, the same definition is valid
for any isometry. In particular, we will allow $\phi$ to be induced on
$\ACm(F)$ by a periodic or reducible automorphism of $F$ as well as by
a pseudo-Anosov.

Elliptic and hyperbolic isometries realize their translation
lengths~\cite{Gromov}, and so $A_{\phi}$ is a nonempty set.  Note also
that $A_{\phi}$ is closed. As we will see from
\autoref*{exists-point-far-from-axis}, $A_{\phi}$ is not the whole
space. To prove this we need the following Lemma, which shows that if
a point $x$ has large translation distance under $\phi$, then $x$ is
far from $A_{\phi}$, and vice versa.

\begin{lemma}\label{translation-distance}
  For every $x\in X$,
  \[ 2d(x, A_{\phi}) + \translen{\phi} - 18\delta\leq d(x, \phi(x)) 
     \leq 2d(x, A_{\phi}) + \translen{\phi}. \]
\end{lemma}
\begin{proof}
  Given $x\in X$, the second inequality is a consequence of the
  triangle inequality: let $y$ be a projection of $x$ to $A_\phi$, so
  that $d(x,y) = d(x, A_{\phi})$. Then $d(x, y) = d(x, A_{\phi})$ and
  $d(y, \phi(y)) \leq \translen{\phi}$, so

  \begin{align*}
    d(x, \phi(x)) &\leq d(x, y) + d(y, \phi(y)) 
    + d(\phi(x), \phi(y)) \\
    & \leq 2d(x, A_{\phi}) + \translen{\phi}.
  \end{align*}  

  To prove the first inequality, we may assume that $x\not\in
  A_{\phi}$.  Let $y$ be a projection of $x$ to $A_{\phi}$.  We claim
  that $[y,\phi(y)]\subseteq A_{\phi}$.  To prove this, choose a point
  $y'\in [y,\phi(y)]$.  Because $\phi$ is an isometry,
  \begin{align*}
    d(y', \phi(y')) &\leq d(y', \phi(y)) + d(\phi(y), \phi(y'))\\
    & = d(y, y') + d(y', \phi(y))\\
    & = d(y, \phi(y)).
  \end{align*}
  Since $y\in A_{\phi}$, $y'$ is also in this set.  Finally, note that
  $\phi(y)$ is a projection of $\phi(x)$ to the geodesic segment
  $[y, \phi(y)]$.

  We claim that $\dist{y}{\phi(y)}\geq 10\delta$.  This is because no
  point $z\neq y$ in $[x,y]$ lies in $A_{\phi}$.  Therefore
  $\dist{z}{\phi(z)} > \max\{\translen{\phi}, 10\delta\}$, and taking
  the limit as $z$ approaches $y$ we obtain the claim.

  By~\autoref*{quad-estimate},
  \begin{align*}
    d(x, \phi(x)) &\geq d(x, y) + d(y, \phi(y)) + d(\phi(x), \phi(y))
    - 18\delta\\
    & = 2d(x,y) + d(y, \phi(y)) - 18\delta\\
    & \geq 2d(x, A_{\phi}) + \translen{\phi} - 18\delta.
  \end{align*}
\end{proof}


The following corollary is immediate:

\begin{corollary}\label{exists-point-far-from-axis}
  If $\phi \co F \to F$ is not the identity then for every $C > 0$,
  there is a vertex $x$ in $\ACm(F)$ such that $\dm(x, A_{\phi}) > C$.
\end{corollary}
\begin{proof}
  As noted in~\cite[Corollary 1.1]{RafiSchleimer}, if an automorphism
  $\phi$ of $F$ induces an isometry of the curve complex of $F$ for
  which there is a universal bound on $d(x, \phi(x))$ then $\phi$ must
  be the identity map.   As the curve complex is quasi-isometric to
  $\ACm(F)$ this argument carries over directly to our case.  Since we
  have assumed $\phi$ is not the identity, the contrapositive implies
  that there must be an $x$ such that $\dm(x, \phi(x)) > 2C +
  \translen{\phi}$. Then by the second inequality in
  \autoref*{translation-distance},
  \[ C < \frac{\dm(x, \phi(x)) - \translen{\phi}}{2} 
     \leq \dm(x, A_\phi). \]
\end{proof}

Before finishing this section, we will need a Lemma relating the
geometric intersection number between two curves in $\ACm(F)$ to their
distance in the complex. An analogous result for loops in the curve
complex is well known, but when arcs are involved the proof becomes
slightly more complex.

Let $x$ and $y$ be properly embedded $1$--manifolds in a surface $F$.
Denote the \defn{geometric intersection number} between $x$ and $y$,
the minimum number of intersections among all $1$--manifolds properly
isotopic to $x$ and $y$, by $\geomint{x}{y}$.  Since we are working in
the marked arc and curve complex, we consider only isotopy disjoint
from the marked points.

\begin{lemma}\label{lem:arc-intersection-number}
  Let $F$ be a connected, orientable surface of genus $g>0$ with one
  or more marked boundary components.  Let $x,y\in\ACm(F)$ be
  represented by arcs in $F$, and assume that $\geomint{x}{y}>0$.
  Then
  \[ \dm(x, y) \leq \geomint{x}{y} + 1. \]
\end{lemma}

Note that unlike the analogous result in the curve complex, the distance 
is bounded by the intersection number rather than the log of the 
intersection number.

\begin{proof}
  Isotope $x$ and $y$ to minimize $|x\cap y|$.  If $\geomint{x}{y}=0$,
  the result holds.  Otherwise, assume $\geomint{x}{y}=n>0$ and suppose 
  that the result holds for all arcs $\alpha, \beta\in\ACm(F)$ with
  $\geomint{\alpha}{\beta}<n$.  We will construct an arc $z\in\ACm(F)$
  with $\geomint{x}{z}=0$ and $\geomint{z}{y}\leq n-1$, which gives
  the result.

  Let $a$ be a longest subarc of $y$ connecting an endpoint of
  $\boundary y$ to a point $p$ of $x\cap y$ whose interior is disjoint
  from $x$.  The point $p$ divides $x$ into two arcs, $x_1$ and
  $x_2$.  We form two new properly embedded arcs $z_1=a\cup x_1$ and
  $z_1=a\cup x_2$.  Note that we may isotope $z_1$ and $z_2$ to be
  disjoint from $x$.  Furthermore, both $\geomint{z_1}{y}$ and
  $\geomint{z_2}{y}$ are less than $n$.  We must show that one of
  $z_1$ and $z_2$ is an essential arc in $F$; this will be our $z$.

  Suppose that $z_i$, $i=1$ or $2$, cobounds a disk with part of
  $\boundary F$.  Then this disk must contain the marked point $m$,
  for otherwise we could use the disk to reduce $|x\cap y|$.
  Therefore $z_j$ cannot also cobound a disk with part of $\boundary
  F$, where $\set{i,j}=\set{1,2}$.
\end{proof}

\begin{lemma}\label{lem:distance-intersection-number}
  Let $F$ be a connected, orientable surface of genus $g>0$ with one
  or more marked boundary components.  Let $x, y\in\ACm(F)$ be vertices
  that are represented by an arc and a \scc in $F$, respectively.
  Then
  \[ \dm(x, y)\leq \geomint{x}{y} + 1. \]
\end{lemma}
\begin{proof}
  By an isotopy of $y$ supported in a neighborhood of $x$, we obtain a
  properly embedded arc $y'$ with $\geomint{y}{y'} = 0$ and
  $\geomint{x}{y'}\leq\geomint{x}{y} - 1$.  The result follows by
  applying~\autoref*{lem:arc-intersection-number}.
\end{proof}

\section{Surfaces compatible with the fibration}
We now switch from considering vertices of $\ACm(F)$ as loops in an
abstract surface to thinking about them as knots in a $3$--manifold
with an open book decomposition.  We find a special surface
$S\subseteq M$ which allows us to give a bound on the translation
distance of the isometry on $\ACm(F)$ induced by the monodromy $\phi$.
Let $K$ be the knot defined by a loop in $F^0$, as before, and write
$\L=K\cup L$.

Recall that we say the link $\L$ is in \defn{bridge position} with
respect to $\Sigma$ if $\Sigma$ divides $M$ into two compression
bodies $H^1$ and $H^2$ such that each arc of $H^i\cap\L$ is trivial in
$H^i$, $i=1,2$.  Say that such a bridge position is \defn{minimal} if
$|H^i\cap\L|$ is minimized over all surfaces isotopic to $\Sigma$. In
this case we define $b_{\Sigma}(\L) = |H^i\cap\L|$.

Note that if $\Sigma$ is in bridge position with respect to $\L$ then
it is in bridge position with respect to each of $K$ and $L$,
independently. Isotope $\Sigma$ so that it is a minimal bridge
surface with respect to $K$. Then isotope it further so that it has
minimal bridge number with respect to $\L = K \cup L$, subject to the
constraint that it remains a minimal bridge surface with respect to
$K$.  We call this a \defn{minimal $K$--bridge position}.

We say that $(\Sigma, \L)$ is \defn{weakly reducible} if there are
disjoint essential disks $D_1\subseteq H^1_{\L}$ and $D_2\subseteq
H^2_{\L}$.  If $(\Sigma, \L)$ is not weakly reducible, we say that it
is \defn{strongly irreducible}.  We use a characterization of
bridge surfaces, due in its original form to Hayashi and
Shimokawa~\cite{Hayashi01}, and reformulated by Taylor and
Tomova~\cite{TaylorTomova13}.  It says that weakly reducible
splittings have one of several properties, or that the exterior of the
link contains a special essential surface.  Let $(\Sigma, \L)$ be a
bridge splitting.  The splitting is \defn{stabilized} if there are
compressing disks on either side of $\Sigma$ that intersect exactly
once.  The splitting is \defn{boundary stabilized} if it is obtained
from a bridge splitting $(\Sigma', \L)$ by amalgamating with the
standard splitting of a neighborhood of a boundary component of $M$.
The splitting is \defn{perturbed} if there are bridge disks (disks
whose boundary consists of the union of an arc of $H^i\cap\L$ and an
arc in $\Sigma$) on opposite sides of $\Sigma$ meeting in a single
point.  Finally, we say that a component $L_0$ of $\L$ is
\defn{removable} if $L_0$ is isotopic to a core of either $H^1$ or
$H^2$.

\begin{theorem}\label{bridge-sfce-classification}
  Let $M$ be a compact, orientable $3$--manifold containing a link
  $\L$, and suppose that $M_{\L}$ is irreducible and that no sphere in
  $M$ intersects $\L$ exactly once.  If $\Sigma$ is a weakly reducible
  bridge surface for $\L$, then one of the following holds:
  \begin{enumerate}
  \item $\Sigma$ is stabilized, boundary stabilized, perturbed, or a
    component of $\L$ is removable, or
  \item $M$ contains an essential meridional surface $S$ such that
    $\chi(\Sigma_{\L})\leq \chi(S_{\L})$.  Furthermore, $S$ is a thin
    surface in a generalized bridge splitting obtained by
    untelescoping $\Sigma$.
  \end{enumerate}
\end{theorem}
\begin{proof}
  This is~\cite[Corollary 9.4]{TaylorTomova13} where we have taken
  $\Gamma=\emptyset$ and $T=\L$.  Note that meridional stabilization
  and boundary meridional stabilization cannot occur as $\Gamma$ is
  empty.
\end{proof}

We wish to give a lower bound on the bridge number of $K$ with respect
to the splitting $\Sigma$. Note the following Lemma, which follows
directly from the definitions:

\begin{lemma}
  Suppose that $(\Sigma', \L)$ is obtained by destabilizing or boundary
  destabilizing the $K$--minimal bridge splitting $(\Sigma, \L)$.  Then
  $b_{\Sigma}(K)\geq b_{\Sigma'}(K)$.  Furthermore, a minimal bridge
  splitting $(\Sigma, \L)$ is never perturbed.
\end{lemma}

By this lemma and~\autoref*{bridge-sfce-classification} we may assume
that either $(\Sigma, \L)$ is strongly irreducible, $M_{\L}$ contains
an essential surface $S_{\L}$ of Euler characteristic greater than or
equal to that of $\Sigma_{\L}$, a component of $L$ is removable, or
$K$ is removable.  We will show that either $\Sigma$ or $S$ can be
isotoped to intersect the fibers $F^s$ in a very controlled manner in
all of these cases.

As noted in~\cite{TaylorTomova13}, if a component $R$ of a link $\L$ is
removable with respect to a Heegaard surface $\Sigma$ then $\Sigma$
can be isotoped so that $R$ is a core of one of the compression bodies
bounded by $\Sigma$. In this case, removing an open regular
neighborhood of $R$ turns the handlebody into a compression body. Thus
after the isotopy, we find a Heegaard surface $\Sigma'$ for the
complement $M_R$ such that $\Sigma'$ is isotopic in $M$ to
$\Sigma$. We consider two cases: when $R$ is a component of $L$
and when $R = K$.

If $R$ is a component of $L$ then the restriction of the bundle map
$\pi\co M \setminus L \to S^1$ to $M_R \setminus L$ is a surface
bundle such that $K$ is contained in a fiber.  We can
apply~\autoref*{bridge-sfce-classification} to the isotoped Heegaard
surface $\Sigma'$ for $M_R$, which is a bridge surface for $\L
\setminus R$. If $R = K$ is the removable component, then we no longer
have a surface bundle structure for $M_R$, but we can still find a
Heegaard surface $\Sigma'$ for $M_R$.

If a component of $\L$ is again removable with respect to $\Sigma'$
then we can repeat this process. We will eventually find a sublink
$L'$ of $\L$ and a Heegaard surface $\Sigma'$ for $M_{L'}$ that is a
bridge surface for $\L \setminus L'$ in which no component is
removable. Note that we may find $L' = \L$, in which case $\Sigma'$ is
just a Heegaard surface and $\L \setminus L'$ is empty.

If $K$ is a component of $L'$ then we let $L'' = L' \setminus K$. Then
$\Sigma'$ is a Heegaard surface for $M_{L''}$ such that $\Sigma'$ is a
bridge surface for $L \setminus L''$ and $K$ is a core of one of the
compression bodies bounded by $\Sigma'$. Thus we have the following:

\begin{lemma}\label{lem:bridge-taxonomy}
  If $\Sigma$ is a bridge surface for $\L$ then there is either
  \begin{enumerate}
  \item a surface $S \subseteq M$ such that $S_{\L}$ is essential and
    $\chi(S_{\L}) \geq \chi(\Sigma_{\L})$,
  \item a sublink $L' \subseteq L$ and a Heegaard surface $\Sigma'$ for
    $M_{L'}$ such that $\Sigma'$ is a strongly irreducible bridge surface for $\L
    \setminus L'$ and $\chi(\Sigma_{\L}') \geq \chi(\Sigma_{\L})$, or
  \item a sublink $L'' \subseteq L$ and a Heegaard surface $\Sigma'$
    for $M_{L''}$ such that $\chi(\Sigma_{\L}') \geq
    \chi(\Sigma_{\L})$, $\Sigma'$ is a strongly irreducible bridge
    surface for $L \setminus L''$ and $K$ is a core of a compression
    body bounded by $\Sigma'$.
  \end{enumerate}
\end{lemma}

Suppose that $M_{\L}$ contains an essential surface $S_{\L}$ such
that $\chi(\Sigma_{\L})\leq \chi(S_{\L})$ as in case one
of~\autoref*{lem:bridge-taxonomy}.
Isotope $S$ so that $\partial S_L \cap \partial F^s_L$ is
minimal for every $s\in S^1$. Moreover, by a general position
argument, we can isotope $S$ so that the restriction of $\pi$ to $S_L$
is Morse and $0$ is a regular value.

\begin{lemma}\label{thin-sfce}
  Let $S_{\L}$ be an essential meridional surface in $M_{\L}$ such that
  $\pi|_{S_L}$ is Morse and $\partial S_L \cap \partial F^s_L$ is
  minimal for every $s\in S^1$.  Then for every regular value $s$ of
  $\pi|_{S_L}$, each arc of $S_L \cap F^s_L$ is essential in both
  surfaces and each \scc of intersection is either essential in both
  surfaces or trivial in both surfaces.
\end{lemma}
Note that the conclusion of the lemma applies to surfaces in the
exterior of $L$ as opposed to $\L$.
\begin{proof}
  In this proof and the sequel we often think of $S^1$ as the interval
  $[0,1]$ with its endpoints identified.  Suppose then that $s\in
  (0,1)$.  Every arc or \scc of $F^s_L\cap S_L$ which is trivial in
  $S_L$ must also be trivial in $F^s_L$ since $F^s_L$ is essential in
  $M_L$.  If there is a trivial \scc or arc of $F^s_L\cap S_L$ in
  $F^s_L$ that is essential in $S_L$, an innermost such \scc or
  outermost such arc bounds or cobounds a disk $D$ in $F^s_L$ disjoint
  from $K$.  Therefore $S_{\L}$ is compressible or
  $\boundary$--compressible, a contradiction.  If such a disk exists
  when $s=0$, there must be some $s>0$ for which there is a similar
  disk in $F^s_L$.

  Finally, note that no arc of intersection can be trivial in both
  surfaces since this would define an isotopy reducing $\boundary
  S_L\cap \boundary F^s_L$.  

\end{proof}

Suppose now that we are in case two or three
of~\autoref*{lem:bridge-taxonomy}, so that $\Sigma'$ is a Heegaard
surface for $M_{L'}$ or $M_{L''}$.  Let $\Lhat$ be $L\setminus L'$ in
case two or $L\setminus L''$ in case three.  Then $\Sigma'$ is either
a strongly irreducible bridge surface for $\Lhat\cup K$ or $\Sigma'$
is a strongly irreducible bridge surface for $\Lhat$ and
$K$ is a core of one of the compression bodies bounded by $\Sigma'$.
We will show that $\Sigma'$ behaves in much the same way as the
essential surface of \autoref*{thin-sfce}.

Let $H_-$ and $H_+$ be the compression bodies bounded by $\Sigma'$ and
let $G_-$, $G_+$ be spines for these compression bodies. If $\Sigma'$
is a bridge surface for $\Lhat\cup K$ then we can extend each of $G_-$
and $G_+$ to contain a single vertex in each arc $\Lhat \cap H_-$ and
$\Lhat \cap H_+$, respectively. Otherwise, if $K$ is a core of one of
$H_-$, $H_+$ then we can choose $G_-$ and $G_+$ so that $K$ is
contained in one of the spines and each arc of $\Lhat \cap H_-$ and
$\Lhat \cap H_+$ contains a vertex of $G_-$ or $G_+$, respectively.

Let $h$ be a sweep-out subordinate to the bridge surface $\Sigma'$,
i.e.\ a function $h\co M\to I$ such that $h^{-1}(0) = G_-$, $h^{-1}(1)
= G_+$, $h^{-1}(\frac{1}{2}) = \Sigma'$, and for each $t \in (0,1)$,
the level surface $\Sigma^t = h^{-1}(t)$ is isotopic to $\Sigma'$ by
an isotopy transverse to $\Lhat\cup K$.  Let $H^t_-$ and $H^t_+$ be
the compression bodies bounded by $\Sigma^t$.

Consider the map $\Phi\co M\setminus L\to S^1\times I$ given by
sending $x\mapsto (\pi(x), h(x))$.  Each point $(s, t)$ in the
cylinder $S^1\times I$ represents a pair of surfaces $F^s$ and
$\Sigma^t$, as described above.  The \defn{graphic} is the subset of
the cylinder consisting of all points $(s, t)$ where $F^s$ and
$\Sigma^t$ are tangent.  We may assume that $\Phi$ is \defn{generic}
in the sense that it is stable (see~\cite{KobayashiSaeki}
and~\cite{Johnson09}) on the complement of $G^+\cup G^-$, each arc
$\set{s}\times I$ contains at most one vertex of the graphic, and each
circle $S^1\times\set{t}$ contains at most one vertex of the graphic.
Vertices in the interior of the graphic are valence four (crossings)
and valence two (cusps).  By general position of $G^+\cup G^-$, the
graphic is incident to the boundary of the cylinder in only a finite
number of points, and each vertex in the boundary has valence one or
two.

For a regular value $(s, t)$ of $\Phi$, say that $F^s$ is
\defn{essentially above} $\Sigma^t$ if there is a component of
$F^s\cap \Sigma^t$ that is an essential arc or circle in $\Sigma^t$
but bounds a compressing or $\boundary$--compressing disk for
$\Sigma^t$ contained in $H^t_-$.  (Note that this compressing disk may
not be contained in $F^s$, though it will often be parallel to a disk
in $F^s$.)  Similarly, say that $F^s$ is \defn{essentially below}
$\Sigma^t$ if there is a component of $F^s\cap\Sigma^t$ that is an
essential arc or circle in $\Sigma^t$ but bounds a compressing or
$\boundary$--compressing disk for $\Sigma^t$ in $H^t_+$.  Let $Q_a$
and $Q_b$ denote the points in $S^1\times (0,1)$ for which $F^s$ is
essentially above or below $\Sigma^t$, respectively.

In the present case $\Sigma'$ is a strongly irreducible bridge surface,
and so we must have $Q_a\cap Q_b = \emptyset$ and, further, no circle
$S^1\times\set{t}$ meets both $Q_a$ and $Q_b$ for any $t\in I$
(\cf~\cite[Lemma 7.3]{BlairEtAl}).  Therefore there is a $t_0\in I$
such that for all $s\in S^1$, $(s, t_0)$ does not meet $Q_a\cup Q_b$
and meets at most one vertex of the graphic.

Recall that a function $f \co S \to S^1$ is \defn{Morse} if every
critical point is non-degenerate and the images in $S^1$ of any two
critical points are distinct. We say that a smooth function $f$
is \defn{almost-Morse} if every critical point is non-degenerate and at
most two critical points are sent to the same level in $S^1$. Note
that according to this definition, Morse functions are also
almost-Morse.

Compare the following lemma to~\cite[Lemma 6.5]{BlairEtAl}, whose
proof is similar.

\begin{lemma}\label{si-bridge-sfce}
  There is a surface $S$ isotopic to $\Sigma'$ such the restriction of
  $\pi$ to $S_L$ is almost-Morse and for every regular value $s \in
  S^1$ of $\pi|_{S_L}$, each arc of $S_L \cap F^s_L$ is essential in
  both surfaces and each \scc of intersection is either essential in
  both surfaces or trivial in both surfaces.
  
  Furthermore, if $c_1$ and $c_2$ are two distinct critical points of
  $\pi|_S$ with $\pi(c_1)=\pi(c_2)=s_0$, then for small $\epsilon >
  0$, each arc of $F^{s_0-\epsilon}_L\cap S_L$ can be
  isotoped rel boundary in $F_L$ to have interior disjoint from the
  interior of each arc of $F^{s_0+\epsilon}_L\cap S_L$.
\end{lemma}
Again, note that this lemma applies to surfaces in the exterior of
$L$, not $\L$.
\begin{proof}
  As noted above, we can choose a value $t_0$ such that for all $s\in
  S^1$, $(s, t_0)$ does not meet $Q_a\cup Q_b$ and meets at most one
  vertex of the graphic. Define $S=\Sigma^{t_0}$.

  The restriction of $\pi$ to $S_L$ defines a function $\pi_0 \co S_L
  \to S^1$. The critical values of $\pi_0$ are the values $s$ such
  that the point $(s, t_0)$ is contained in an edge of the
  graphic. Points in the interior of edges correspond to
  non-degenerate critical points, with vertices corresponding to two
  critical points whose images in $S^1$ coincide. (Cusp points
  correspond to degenerate critical points in $\pi_0$.) Because of the
  way we chose $t_0$, we conclude that $\pi_0$ is Morse (if there is
  no vertex $(s, t_0)$) or almost-Morse.

  Suppose first that $s\in(0, 1)$.  Note that every arc or \scc of
  $F^s_L\cap S_L$ which is trivial in $S_L$ must also be trivial in
  $F^s_L$ since $F^s_L$ is essential.  Suppose then that there is a
  trivial \scc or arc of $F^s_L\cap S_L$ in $F^s_L$ that is essential
  in $S_L$.  An innermost such \scc or arc bounds or cobounds a disk
  $D$ in $F^s_L$.  Isotope $D$ fixing $\boundary D\cap S_L$ so that
  $|D\cap S_L|$ is minimal.  An innermost disk argument shows that we
  may take the interior of $D$ disjoint from $S_L$, contradicting the
  assumption that $(s, t_0)$ does not meet $Q_a\cup Q_b$ for any $s\in
  S^1$.  If $s=0$, so that $D$ lies in $F^0_L$, then there must be
  some $s>0$ for which there is a similar disk in $F^s_L$.  Therefore,
  for every regular value $s \in S^1$ of $\pi|_{S_L}$, each arc of
  $S_L \cap F^s_L$ is essential in both surfaces and each \scc of
  intersection is either essential in both surfaces or trivial in both
  surfaces.

  If $\pi_0$ is Morse (so that $(s, t_0)$ does not meet a vertex of
  the graphic for any $s\in S^1$), we proceed as in the proof
  of~\autoref*{thin-sfce}.  We are left to show that when $\pi_0$ is
  almost-Morse with two critical points $c_1$, $c_2$ such that
  $\pi(c_1)=\pi(c_2)=v$ then we can pair arcs of
  $F^{s_0-\epsilon}_L\cap S_L$ with disjoint (in $F_L$) arcs of
  $F^{s_0+\epsilon}_L\cap S_L$ for small $\epsilon$.  As noted above,
  this case will only occur if $(s_0, t_0)$ is a valence-four vertex
  of the graphic in the intersection of the closures of $Q_a$ and
  $Q_b$.  Let $\epsilon >0$ be small enough so that this is the only
  vertex with $t$-value in the interval $[t_0-\epsilon,
  t_0+\epsilon]$.  Let $s_-=s_0-\epsilon$ and $s_+=s_0+\epsilon$.

  Therefore we must show that each component of $\Wbf=F^{s_-}_L\cap
  S^{t_0}_L$ can be isotoped in $F_L$ to be disjoint from each
  component of $\Ebf=F^{s_+}_L\cap S^{t_0}_L$.  Going from
  $\Nbf=F^{s_0}_L\cap S^{t_0+\epsilon}_L$ to $\Sbf=F^{s_0}_L\cap
  S^{t_0-\epsilon}$, we pass through two saddles of $F_L$.  This
  corresponds to adding two bands, $b_1$ and $b_2$, to components of
  $\Nbf$ to obtain the new components of $\Sbf$.  If ends of $b_1$ and
  $b_2$ are adjacent to different components of $\Nbf$ or if ends of
  $b_1$ and $b_2$ are adjacent to the same side of the same component,
  then by isotoping those components slightly we see that $\Nbf$ can
  be made disjoint from $\Sbf$ in $S_L$.  This contradicts the strong
  irreducibility of $S$ and implies in particular that the critical
  set is connected.  This in turn implies that at most 3 isotopy
  classes of arcs differ between $F^{s_-}_L\cap S_L$ and
  $F^{s_+}_L\cap S_L$.

  \autoref*{fig:saddleswitch} shows two examples that violate the
  strong irreducibility of $S$.  In both cases the critical set is a
  tree.  Pictured above is the surface $S$ in a regular neighborhood
  $B\subseteq M$ of this tree.  The intersections $F^s\cap B$ form a
  family of parallel horizontal disks.  Intersections of these disks
  with $S$ are shown at the critical level as well as just before and
  just after the critical level.  The projection of the arcs into $F$
  is shown below each $3$--dimensional picture.  

  \begin{figure}[htb]
    \begin{center}
      \includegraphics[width=4.5in]{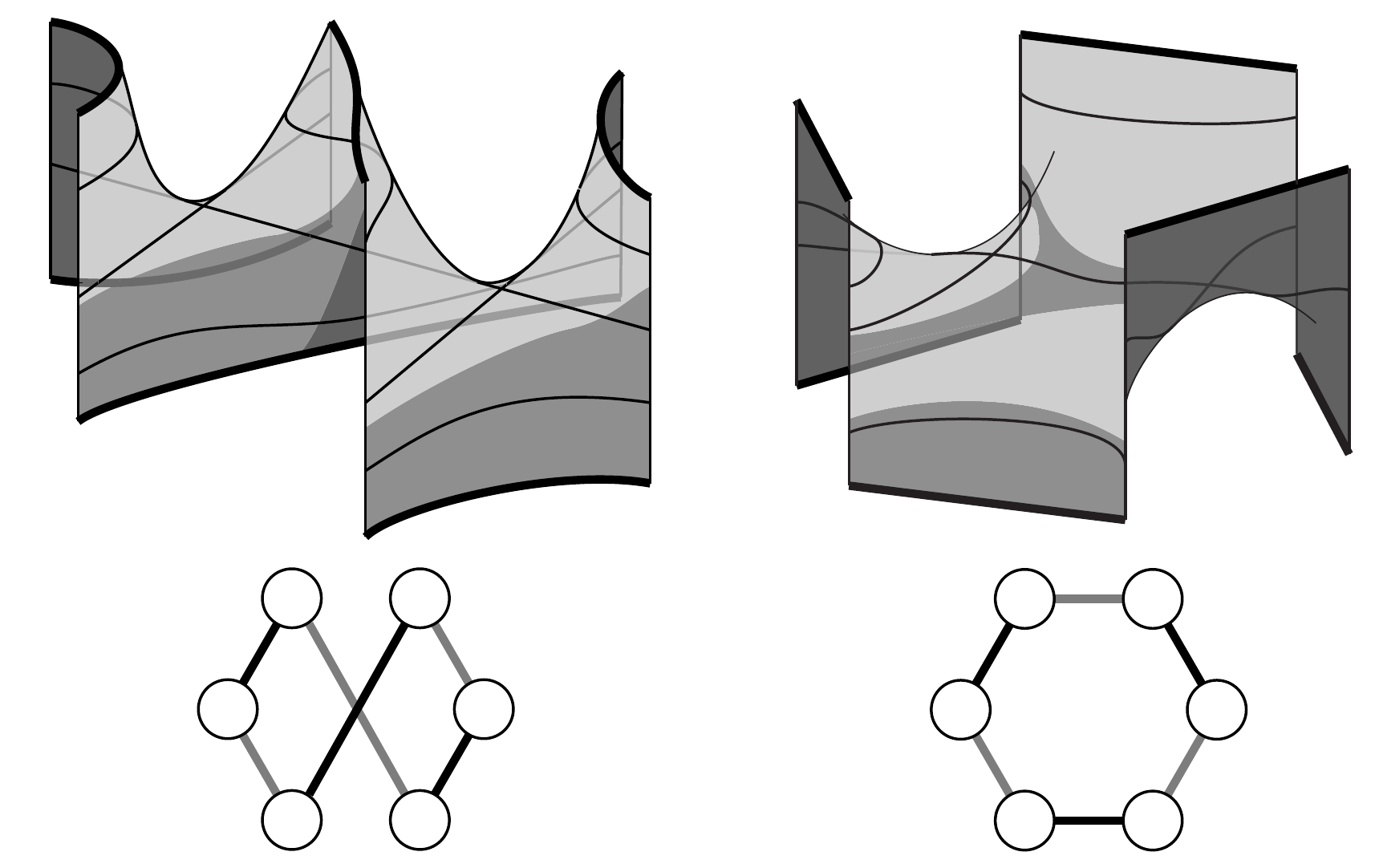}
      \caption{Two cases which violate the strong irreducibility of
      $S$.}\label{fig:saddleswitch}
    \end{center}
  \end{figure}

  From the previous discussion we see that an end of $b_1$ and an end
  of $b_2$ are adjacent to opposite sides of the same component
  $\alpha$ of $\Nbf$.  If the other ends of $b_1$ and $b_2$ are
  adjacent to a component $\alpha'$ of $\Nbf$, then those ends must be
  adjacent to opposite sides of $\alpha'$.  If they were not, one of
  $b_1$ or $b_2$ would be a band with ends attached to opposite sides
  of a component of $\Ebf$ or $\Wbf$, contradicting the orientability
  of $F$ and $\Sigma$.  Thus the arcs and \sccs that result from
  attaching $b_1$ to $\Nbf$ are disjoint from the arcs and \sccs that
  result from attaching $b_2$ to $\Nbf$ since we may push each such
  component slightly in the direction that the bands approach from.
  This shows that each component of $\Ebf$ can be made disjoint from
  each component of $\Wbf$ in $F_L$, and so each arc of
  $F^{s_0-\epsilon}_L\cap S_L$ can be isotoped rel boundary in $F_L$ to have
  interior disjoint from the interior of each arc of
  $F^{s_0+\epsilon}_L\cap S_L$.
\end{proof}

The next proposition follows by putting
together~\autoref*{lem:bridge-taxonomy}, \autoref*{thin-sfce}, and
\autoref*{si-bridge-sfce}.  Note first that if $S$ is an essential
surface, it is a thin surface in a generalized bridge splitting
by~\autoref*{bridge-sfce-classification}, and so
$b_{\Sigma}(K)\geq\frac{1}{2}|K\cap S|$.  Second, the surface $S$ is
obtained from $\Sigma$ by compression (possibly zero times), so
$g(S)\leq g(\Sigma)$.

\begin{prop}\label{exists-nice-surface}
  There is a surface $S\subseteq M$ such that the following properties
  hold:
  \begin{enumerate}
  \item $g(S)\leq g(\Sigma)$ 
  \item $b_{\Sigma}(K)\geq \frac{1}{2}|K\cap S|$,
  \item the restriction of $\pi$ to $S_L$ is Morse or almost-Morse,
  \item every arc of $S_L \cap F^s_L$ is essential in both surfaces for 
    regular values of $s \in S^1$ and,
  \item if $c_1$ and $c_2$ are two distinct critical points of
    $\pi|_{S_L}$ with $\pi(c_1)=\pi(c_2)=s_0$, then the arcs of
    $F^{s_0-\epsilon}_L\cap S_L$ may be isotoped rel boundary in $F_L$
    to have interiors disjoint from the interiors of
    $F^{s_0+\epsilon}_L\cap S_L$ for small $\epsilon$.
  \end{enumerate}    
\end{prop}

\section{Distance bounds}
Recall that $L$ is the binding of an open book decomposition of the
closed, orientable, connected manifold $M$ with page $F$, $\phi$ is
the monodromy of the open book, and $A_{\phi}$ is the axis of $\phi$
as an automorphism of $\ACm(F)$.  By~\autoref*{exists-nice-surface}
there is a ``nice'' surface $S$ which we will use now.  Compare the
following lemma with~\cite[Theorem 5.3]{BlairEtAl}, whose proof is
similar.

\begin{lemma}\label{lem:distance-bound}
  There is an arc $x$ of $F_L\cap S_L$ whose class in $\ACm(F)$ (which
  we also write as $x$) satisfies
  \[ d(x, \phi(x)) \leq \frac{-4\chi(S_L)}{|\boundary S_L\cap
    \boundary F_L|}. \]
\end{lemma}
\begin{proof}
  The number of arcs of $F^s_L\cap S_L$ is constant for all regular values
  $s\in S^1$ of $\pi|_S$, namely equal to $n=|\boundary S_L\cap\boundary
  F_L|/2$.  Recall that we have arranged for $0$ to be a regular value
  of $\pi|_S$, and think of $S^1$ as the interval $I=[0,1]$ with its
  endpoints identified.  Suppose that $c\in S$ is an index-one critical
  point of $\pi|_S$ and that $v=\pi(c)\in I$ is the associated
  critical value.  Choose $s_-<v<s_+$ so that $v$ is the only 
  index-one critical value in $[s_-, s_+]$.  As $v$ goes from $s_-$ to
  $s_+$, arcs and circles of $F^{s_-}_L\cap S_L$ are banded together by
  one or two bands to obtain curves isotopic to $F^{s_+}_L\cap S_L$.  


  Let $A_{\pm}$ be the union of the arc components of
  $F^{s_{\pm}}_L\cap S_L$ adjacent to a band whose other end
  is adjacent to either an arc of $F^{s_{\pm}}_L\cap S_L$ or to a \scc
  of $F^{s_{\pm}}_L\cap S_L$ that is essential in $S$.  The
  components of $A_-$ are called \defn{pre-active arcs} and the
  components of $A_+$ are called \defn{post-active arcs}.  An arc that
  is pre- or post-active is called \defn{active}.  A circle component
  of $F^{s_{\pm}}_L\cap S_L$ to which a band or bands are adjacent is
  called an \defn{active circle} if it is essential in $S_L$.  If
  $A_{\pm}\neq\emptyset$, then we call $v$ an \defn{active index-one
    critical value} and say that the index-one critical points in its
  preimage are \defn{active index-one critical points}.  

  The surface $F_L$ is essential in $M_L$, so for regular values $s$,
  no component of $F^s_L\cap S_L$ is essential in $F^s_L$ and
  inessential in $S_L$.  Therefore each arc component of
  $F^{s_-}_L\cap S_L$ that is not pre-active is isotopic in $F_L$ to
  an arc component of $F^{s_+}_L\cap S_L$ that is not
  post-active. Thus there is a bijection between the arc components of
  $F^{s_-}_L\cap S_L$ and $F^{s_+}_L\cap S_L$ taking $A_-$ to $A_+$
  which is constant on the isotopy classes of arcs in $F^{s_-}_L$ that are
  not in $A_-$.

  By~\autoref*{exists-nice-surface} there is a pairing of arcs between
  $F^{s_-}_L\cap S_L$ and $F^{s_+}_L\cap S_L$ such that each arc is
  distance one from its paired arc in $\ACm(F)$.  Therefore we may
  construct paths in $\ACm(F)$ from the isotopy classes of arcs in
  $F^0_L\cap S_L$ to isotopy classes of arcs in $F^1_L\cap S_L$.  Let
  $q_i$ be the length of the $i$-th path, and note that as
  $\dm(F^0_L\cap S_L, F^1_L\cap S_L)$ gives the length of the shortest
  path between these two sets,
  \[ n\cdot\dm(F^0_L\cap S_L, F^1_L\cap S_L) \leq \sum_{i=1}^nq_i. \]

  We wish to bound the sum $\sum_{i=1}^n q_i$.  By the above remarks,
  each time $v$ passes through a critical value of $\pi$, arcs and
  circles are banded together to obtain new arcs and circles.  When
  $v$ is a critical value that has a single associated critical point
  $c$, at most two arcs are banded together.  When $v$ is a critical
  value with two associated critical points $c_1$ and $c_2$, at most
  three arcs are banded by \autoref*{si-bridge-sfce}.  Therefore if
  $V_s$ is the number of critical values that have a single associated
  critical point and $V_d$ is the number of critical values that have
  two associated critical points, then
  \begin{equation}\label{eq:path-sum}
    \sum_{i=1}^n q_i \leq 2V_s + 3V_d. 
  \end{equation}

  Let $P$ be the closure of the complement of the active arcs and
  circles in $S_L$, and denote its components by $P_1,\dots,P_m$.  The
  boundary of each $P_k$ is the union of arcs and \sccs contained in
  $\boundary M_L$ and active arcs and \sccs.   Each $P_k$ contains
  zero, one, or two active index-one critical points of $\pi|_{S_L}$,
  and there is at most one $P_k$ that contains two active index-one
  critical points.  Let $b_k$ be the number of active arcs in
  $\boundary P_k$, and define the \defn{index} of $P_k$ to be
  \[ J(P_k) = b_k/2 -\chi(P_k). \]
  Since each active arc shows up twice in $\boundary P$ and since
  Euler characteristic increases by one when cutting along an arc, we
  have
  \begin{equation}\label{eq:index-sum}
    -\chi(S_L) = \sum_k J(P_k).
  \end{equation}


  Fix $k$.  By hypothesis $P_k$ is not $S^2$.  If $P_k$ is a disk, its
  boundary cannot be an active circle or contain only a single active
  arc as active circles and arcs are essential in $S$.  Thus, if $P_k$
  is a disk, $J(P_k)\geq 0$.  If $P_k$ is not a disk, $-\chi(P_k)\geq
  0$.  It follows that if $P_k$ does not contain an active index-one
  critical point, then its index is nonnegative.  

  Suppose $P_k$ contains a unique active index-one critical point
  $c\in P_k$ and let $\alpha$ be a pre-active arc at $c$.  If $\alpha$
  is banded to a different pre-active arc, then $b_k\geq 4$ since
  there must be at least two pre-active arcs and two post active arcs.
  Then $J(P_k)\geq 1$.  Otherwise let $\gamma$ be the circle that is
  either banded to $\alpha$ or that results from banding $\alpha$ to
  itself.  In this case, $\gamma$ is essential, so $P_k$ is not a
  disk.  There are two active arcs in $\boundary P_k$, so
  $J(P_k)\geq 1$.  It follows that if $P_k$ contains a single active
  index-one critical point, then $J(P_k)\geq 1$.

  If $P_k$ contains two active index-one critical points, they must
  have the same height and be connected.  Let $C_-$ and $C_+$ be the
  number of pre-active and post-active circles, respectively, at $v$.
  Note that any pre- or post-active circle, along with at least one
  post-active arc at $v$, lies in  $\boundary P_k$.  Furthermore, the
  bands must lie in $P_k$ since the bands themselves contain the 
  index-one critical points.  Thus, if $P_k$ is a disk, then $C_-=C_+=0$ 
  and $b_k=6$.  If $P_k$ is not a disk, then either $C_-+C_+\geq 1$,
  $b_k\geq 2$, or $b_k=2$ and $P_k$ is not planar.  In any of these
  cases, $J(P_k)\geq 2$.  Consequently,
  \begin{equation}\label{eq:index-sum-inequality}
    \sum_k J(P_k)\geq V_s + 2V_d.
  \end{equation}

  Putting together~\autoref*{eq:path-sum}, \autoref*{eq:index-sum},
  and \autoref*{eq:index-sum-inequality}, we have
  \begin{align*}
    n\cdot \dm(F^0_L\cap S_L, F^1_L\cap S_L) &\leq \sum_{i=1}^nq_i\\
    & \leq 2V_s + 3V_d\\
    & \leq 2(V_s + 2V_d)\\
    & \leq 2\sum_k J(P_k)\\
    & = -2\chi(S_L),
  \end{align*}
  and since $n=|\boundary F_L\cap\boundary S_L|/2$ is the total number
  of arcs, the result follows.
\end{proof}

\begin{proof}[Proof of~\autoref*{theorem:main}]
  Fix integers $D>0$ and $g(M)\leq g\leq -\chi(F)$.
  By~\autoref*{exists-point-far-from-axis} there is a knot $K\subseteq
  F$ so that $d(K, A_{\phi}) > C$, where
  $C=D-4\chi(F)-\frac{1}{2}\translen{\phi} + 9\delta$.  Let $\Sigma$
  be a genus $g$ bridge surface for $\L$ that is $K$--minimal.

  By~\autoref*{exists-nice-surface} there is a surface $S$ such that
  $g(S)\leq g(\Sigma)$ and every arc of $F^s_L\cap S_L$ is essential
  in both surfaces for regular values of $s$.  We can
  apply~\autoref*{lem:distance-bound} to show that there is an arc $x$
  of $F_L\cap S_L$ whose class in $\ACm(F)$ satisfies
  \begin{align*}
    \dm(x, \phi(x)) &\leq \frac{-4\chi(S_L)}{|\boundary S_L\cap\boundary
      F_L|}\\
    &\leq 8g(S)-4\\
    &\leq 8g(\Sigma) - 4\\
    &\leq -8\chi(F) - 4.
  \end{align*}
  
  By~\autoref*{translation-distance}, we see that for this $x$,
  \[ \dm(x, A_{\phi})\leq \frac{1}{2}(-8\chi(F) - 4 - \translen{\phi}) +
  9\delta. \]
  Therefore,
  \begin{align*}
    \dm(x, K) &\geq \dm(K, A_{\phi}) - \dm(x, A_{\phi})\\
    & > C + 4\chi(F) + 2 + \frac{1}{2}\translen{\phi} - 9\delta\\
    & = D + 2.
  \end{align*}
  
  Using the inequality given
  by~\autoref*{lem:distance-intersection-number}, we see that
  \begin{align*}
    \geomint{x}{K} &\geq \dm(x, K) - 1\\
    & > D.
  \end{align*}
  Since $x$ is an arc in $S_L$, this implies that $|K\cap S| > D$.

  Recall that $S$ is either the bridge splitting $\Sigma$ or a thin
  surface in a generalized bridge splitting for $\L$ obtained by
  untelescoping $\Sigma$.  In either case, we have
  \[ b_{\Sigma}(K) > D. \]
  
  Since $\Sigma$ was arbitrary, it follows that $b_g(K) > D$.  

  Finally, suppose that $M=S^3$ and note that $K$ is not a torus knot
  from the bridge number bounds above.  Suppose that $M_K$ contains an
  essential torus.  By an isotopy minimizing $|T\cap L|$ and keeping
  $T$ disjoint from $K$, we obtain an essential punctured torus in
  $M_{\L}$.  An argument similar to the one above (with $T$ taking the
  place of $S$) shows that $T$ must meet $K$.  This is impossible,
  and so $M_K$ is atoroidal.  Therefore $K$ is a hyperbolic knot in
  $S^3$. 
\end{proof}

\bibliographystyle{plain}
\bibliography{knots-in-page}
\end{document}